\documentclass[a4paper,11pt]{article}
\usepackage{amsmath}
\usepackage{amsfonts}
\usepackage{amssymb}
\usepackage{indentfirst,latexsym,bm}
\usepackage{amsthm}
\usepackage[all]{xy}
\numberwithin{equation}{section}
\newtheorem{thm}{Theorem}[section]
\newtheorem{cor}[thm]{Corollary}
\newtheorem{lem}[thm]{Lemma}
\newtheorem{prop}[thm]{Proposition}
\theoremstyle{definition}
\newtheorem{defn}{Definition}[section]
\theoremstyle{remark}
\newtheorem{rem}{Remark}[section]
\newtheorem{ex}{Example}[section]

\begin{document}


\title{\Large Representations for weighted Moore-Penrose inverses of partitioned
adjointable operators}\author{Qingxiang Xu\thanks{Corresponding
author. Department of Mathematics, Shanghai Normal University,
Shanghai 200234, P.R.~China, and School of Science, Shanghai Institute of Technology, Shanghai, 201418, P.R.~China (qingxiang$\_$xu@126.com, qxxu@shnu.edu.cn). Supported by the
National Natural Science Foundation of China under grant 11171222, and the Innovation Program of Shanghai Municipal
Education Commission under grant 12ZZ129.}\and Yonghao Chen\thanks{Department of Mathematics, Shanghai Normal
University, Shanghai 200234, P.R.~China
(chenyonghao6232@163.com).}\and Chuanning Song\thanks{Department of
Mathematics, Shanghai Normal University, Shanghai 200234, P.R.~China
(songning1962@163.com, songning@shnu.edu.cn).}}
\date{{}}
\maketitle

\vspace{-8ex}

\def\abstractname {}
\begin{abstract}
\noindent\textbf{Abstract}

\vspace{2ex} For two positive definite adjointable operators $M$ and
$N$, and an adjointable operator $A$ acting on a Hilbert
$C^*$-module, some properties of the weighted Moore-Penrose inverse
$A^\dag_{MN}$ are established. If $A=(A_{ij})$ is $1\times 2$ or
$2\times 2$ partitioned, then general representations for $A^\dag_{MN}$
in terms of the individual blocks of $A_{ij}$ are studied. In the case when $A$
is $1\times 2$ partitioned, a unified representation for
$A^\dag_{MN}$ is presented. In the $2\times 2$ partitioned case, an approach to the construction of the Moore-Penrose inverse
 from the
non-weighted case to the weighted case is provided. Some results
known for matrices are extended to the general setting of operators on Hilbert
$C^*$-modules.

\vspace{2ex}


\vspace{2ex}

\end{abstract}

\section*{Introduction}
The weighted Moore-Penrose inverse of an arbitrary  (singular and
rectangular) matrix has many applications in the weighted linear
least-squares problems, statistics, neural network, numerical
analysis and so on. For a partitioned matrix $A=(A_{ij})$, it has
been of interest to derive general expressions for the weighted
Moore-Penrose inverse of $A$ in terms of the individual blocks of
$A_{ij}$. If $A=(A_{11}, A_{12})$  is  a $1\times 2$ partitioned
matrix, then some formulas for the (non-weighted) Moore-Penrose inverse
$A^\dag$, such as Cline \cite{cline} and Mihalyffy \cite{miha} are
well-known. In the weighted case, a formula for $A^\dag_{MN}$ of a
$1\times 2$ partitioned matrix $A$ was given by Miao \cite{miao1}.
Later, this formula  was reproved by Chen \cite{chen}, Wang and
Zheng \cite{wang2} by using different methods. Recently, another
formula for $A^\dag_{MN}$ has been obtained by the first author
\cite{xu1}. In this paper, in the general context of Hilbert
$C^*$-module operators, we will provide a unified representation for
$A^\dag_{MN}$ (see Theorem~\ref{thm:unified representation} below).
As a result, the equivalence of the formulas for $A^\dag_{MN}$ given
respectively  in \cite{miao1} and \cite{xu1} is derived.

If $A=\Big(
          \begin{array}{cc}
            A_{11} & A_{12} \\
            A_{21} & A_{22} \\
          \end{array}
        \Big)$ is a $2\times 2$ partitioned matrix, then things may
        become much more complicated. Most works in literature concerning
        representations for $A^\dag$ were carried out under certain
        restrictions on the blocks of $A_{ij}$. In 1991, a general
        expression for $A^\dag$ without any restriction imposed on
        the blocks of $A_{ij}$, was given by Miao in \cite{miao2}. Since then, more than twenty years has passed. However, due to the
        complexity revealed in \cite{hartwig, meyer, miao2}, there has not  been much progress concerning
        the generalization of Miao's result \cite{miao2} from the non-weighted case to the
        weighted case. In this paper, we
        make such an effort in the general setting of Hilbert $C^*$-module operators.

The paper is organized as follows. In Section~\ref{sec:defn of
weighted MP inverse}, in the general setting of Hilbert $C^*$-module
operators, we will establish some properties on weighted
Moore-Penrose inverses. Following the line initiated in \cite{xu2},
in Section~\ref{sec:relationship of w-m-p} we will study the
relationship between weighted Moore-Penrose inverses $A^\dag_{MN}$,
where $A$ is fixed, while $M$ and $N$ are variable. In
Section~\ref{sec:1x2} we will study unified representations for
weighted Moore-Penrose inverses of $1\times 2$ partitioned
adjointable operators. In Section \ref{sec:weighted MP-inverse of
2x2 matrices}, an approach, initiated in \cite{xu1} for $1\times 2$
partitioned adjointable operators, is applied to study the
general expressions for weighted Moore-Penrose inverses of $2\times
2$ partitioned adjointable operators. Our key point is the
construction of a commutative diagram in page \pageref{commutative
diagram}, through which the main results of \cite{miao2} are
generalized from the non-weighted case to the weighted case.

\section{Weighted Moore-Penrose
inverses of adjointable operators}\label{sec:defn of weighted MP
inverse} In this section, in a general setting of adjointable
operators on Hilbert $C^*$-modules, we establish some properties on
weighted Moore-Penrose inverses, most of which are known for
matrices. Throughout this paper, $\mathfrak{A}$ is a $C^*$-algebra,
$\mathbb{C}$ is the complex field, and $\mathbb{C}^{m\times n}$ is
the set of $m\times n$ complex matrices. By a projection, we mean an
idempotent and a self-adjoint element in a certain $C^*$-algebra. For
any Hilbert $\mathfrak{A}$-modules $H$ and $K$, let ${\cal L}(H,K)$
be the set of \emph{adjointable} operators from $H$ to $K$. If
$H=K$, then ${\cal L}(H,H)$, which we abbreviate to ${\cal L}(H)$, is a
unital $C^*$-algebra, whose unit is denoted by $I_H$. For any $A\in
{\cal L}(H,K)$, the range and the null space of $A$ are denoted by
${\cal R}(A)$ and ${\cal N}(A)$, respectively.

Throughout, the notations of ``$\oplus$" and ``$\dotplus$" are used
with different meanings. For any Hilbert $\mathfrak{A}$-modules
$H_1$ and $H_2$, let
$$H_1\oplus H_2=\bigg\{\binom{h_1}{h_2}\,\bigg|\,h_i\in H_i, i=1,2\bigg\},$$ which is also a Hilbert $\mathfrak{A}$-module whose
$\mathfrak{A}$-valued inner product is given by
$$\bigg<\binom{x_1}{y_1}, \binom{x_2}{y_2}\bigg>=\big<x_1,x_2\big>+\big<y_1,
y_2\big>,\ \mbox{for any $x_i\in H_1$ and $y_i\in H_2, i=1,2.$}$$ If both
$H_1$ and $H_2$  are  submodules of a Hilbert
$\mathfrak{A}$-module $H$ such that $H_1\cap H_2=\{0\}$, then we
define
$$H_1\dotplus H_2=\{h_1+h_2\big|\,h_i\in H_i, i=1,2\}\subseteq H.$$
If furthermore $H=H_1\dotplus H_2$, then we call $P_{H_1,H_2}$ the
\emph{oblique projector} along $H_2$ onto $H_1$, where $P_{H_1,H_2}$
is defined by
$$P_{H_1,H_2}(h)=h_1,\ \mbox{for any}\ h=h_1+h_2\in H\ \mbox{with}\
h_i\in H_i,i=1,2.$$

\begin{lem}\label{lem:orthogonal} {\rm (cf.\,\cite[Theorem 3.2]{lance} and \cite[Remark 1.1]{xu3})}\ Let $H, K$ be two Hilbert $\mathfrak{A}$-modules and $A\in {\cal
L}(H,K)$. Then the closeness of any one of the following sets
implies the closeness of the remaining three sets:
$${\cal R}(A), \ {\cal R}(A^*), \ {\cal R}(AA^*),  {\cal R}(A^*A).$$
Furthermore, if ${\cal R}(A)$ is closed, then ${\cal R}(A)={\cal
R}(AA^*)$, ${\cal R}(A^*)={\cal R}(A^*A)$ and with respect to the $\mathfrak{A}$-valued inner product, the following
orthogonal decompositions  hold:
\begin{equation}\label{equ:orthogonal decomposition} H={\cal N}(A)\dotplus
{\cal R}(A^*), \ K={\cal R}(A)\dotplus {\cal N}(A^*).\end{equation}
\end{lem}

 Throughout the rest
of this section, $H,K$ and $L$ are three Hilbert
$\mathfrak{A}$-modules.

\begin{defn} An element $M$ of ${\cal L}(K)$ is said to be \emph{positive definite}, if $M$ is positive and invertible in ${\cal L}(K)$.
\end{defn}

\begin{prop} Let $M\in {\cal L}(K)$ be
positive definite. Then with the inner-product given by
\begin{equation}\label{equ:M-inner product} \big<x,y\big>_M=\big<x,
My\big>, \ \mbox{for any}\ x,y\in K,\end{equation}  $K$ also becomes
a Hilbert $\mathfrak{A}$-module.
\end{prop}
\begin{proof} With respect to $\big<\cdot\,,\,\cdot\big>_M$, $K$ is clearly an
inner-product $\mathfrak{A}$-module \cite[P.\,2]{lance}. We prove
that $K$ is complete with respect to the norm induced by
\begin{equation}\label{equ:defn of K-norm}\Vert
x\Vert_M\stackrel{def}{=}\Vert\big<x,x\big>_M\Vert^{\frac12}=\Vert
M^{\frac12}x\Vert, \ \mbox{for any $x\in K$.}\end{equation}  In
fact, if we let $C_1=\Vert M^{-\frac12}\Vert^{-1}>0$ and $C_2=\Vert
M^{\frac12}\Vert>0$, then by (\ref{equ:defn of K-norm}) we can get
\begin{equation*}\label{equ:equivalent norm}C_1\ \Vert x\Vert \le \
\Vert x\Vert_M\ \le C_2\ \Vert x\Vert, \ \mbox{for any}\ x\in
K,\end{equation*} which means that $\Vert\cdot\Vert$ and
$\Vert\cdot\Vert_M$ are equivalent norms on $K$. Since $K$ is
assumed to be complete with respect to the original norm
$\Vert\cdot\Vert$, the completeness of $K$ with respect to the
induced norm $\Vert\cdot\Vert_M$ follows.
\end{proof}

\begin{rem} We use the notation $K_M$ to denote the Hilbert $\mathfrak{A}$-module with the
inner-product given by (\ref{equ:M-inner product}), and call $K_M$
the \emph{weighted space} (with respect to $M$). Following the
notation, for any positive definite element $N$ of ${\cal L}(H)$,
$T\in {\cal L}(H,K)$, $x\in H$ and $y\in K$, we have
\begin{eqnarray*}&& \big<Tx, y\big>_M=\big<Tx,My\big>=\big<x,
T^*My\big>\nonumber=\big<x, N^{-1}T^*My\big>_N.\end{eqnarray*} So if
we regard  $T$ as an element of ${\cal L}(H_N, K_M)$\footnote{\ The
reader should be aware that as sets, ${\cal L}(H,K)$ and ${\cal
L}(H_N,K_M)$ are the same.}, then
\begin{equation}\label{equ:weighted star}T^\#=N^{-1}\,T^*\, M,\end{equation}
where $T^\#\in {\cal L}(K_M,H_N)$ is the adjoint operator of
$T\in {\cal L}(H_N,K_M)$.
\end{rem}

\begin{defn}
Let $A\in {\cal L}(H,K)$ be arbitrary, and let $M\in {\cal L}(K)$ and $N\in
{\cal L}(H)$ be two positive definite operators. The \emph{weighted
Moore-Penrose inverse} $A^\dag_{MN}$ (if it exists) is the element
$X$ of ${\cal L}(K,H)$, which satisfies
\begin{equation}\label{equ:defn of WPR inverse} AXA=A, XAX=X,
(MAX)^*=MAX\ \mbox{and}\ (NXA)^*=NXA.\end{equation} If $M=I_K$
and $N=I_H$, then $A^\dag_{MN}$ is denoted simply by $A^\dag$, which
is called the \emph{Moore-Penrose inverse} of $A$.
\end{defn}

\begin{thm} Let $A\in {\cal
L}(H,K)$ be arbitrary, and let $M\in {\cal L}(K)$ and $N\in {\cal L}(H)$ be
two positive definite operators. Then $A^\dag_{MN}$ exists if and
only if $A$ has a closed range.
\end{thm}
\begin{proof} If $A^\dag_{MN}$ exists, then
$AA^\dag_{MN}$ is an idempotent, so ${\cal R}(A)={\cal
R}(AA^\dag_{MN})$ is closed. Conversely, suppose that ${\cal R}(A)$
is closed in $K$, then ${\cal R}(A)$ is also closed in $K_M$, so by
\cite[Theorem~2.2 and Proposition 2.4]{xu3} there exists uniquely an
element $X\in {\cal L}(K_M,H_N)$ satisfying
\begin{eqnarray*}\label{equ:defnnnn of WPR inverse 1}&& AXA=A, \ XAX=X,
(AX)^\#=AX \ \mbox{and}\ (XA)^\#=XA.\end{eqnarray*} By
(\ref{equ:weighted star}) we get $(AX)^\#=M^{-1}(AX)^*M$ and
$(XA)^\#=N^{-1}(XA)^*N$. As $M$ and $N$ are self-adjoint, the last
two equalities in (\ref{equ:defn of WPR inverse}) hold.
\end{proof}

\begin{rem} Let $A\in {\cal
L}(H,K)$  have a closed range, and let $M\in {\cal L}(K)$, $N\in {\cal
L}(H)$ be positive definite. As in the finite-dimensional case
\cite[Theorem 1.4.4]{wang1}, by \cite[Theorem 2.2]{xu3} we have
\begin{eqnarray*}\label{eqn:range1}&&
{\cal R}(A^\dag_{MN})={\cal R}(A^\#)={\cal R}(N^{-1}A^*M)=N^{-1}{\cal R}(A^*),\\
\label{eqn:range2}&&{\cal N}(A^\dag_{MN})={\cal N}(A^\#) ={\cal
N}(N^{-1}A^*M)=M^{-1}{\cal N}(A^*).
\end{eqnarray*}
\end{rem}

\begin{prop}{\rm (cf.\,\cite[Lemma 0.1]{chen})}\label{prop:m-p inverse***}\ Let $A\in {\cal
L}(H,K)$  have a closed range, and let $M\in {\cal L}(K)$ and $N\in {\cal
L}(H)$ be positive definite. Then $A^\dag_{MN}$ is the unique
element $X$ of ${\cal L}(K,H)$ which satisfies
\begin{equation}\label{equ:unique two2} A^*MAX=A^*M, \ {\cal R}(NX)\subseteq {\cal R}(A^*).\end{equation}
\end{prop}
\begin{proof} By \cite[Proposition 2.4]{xu3} we know that
$A^\dag_{MN}$ is the unique element $X$ of ${\cal L}(K_M,H_N)$ which
satisfies \begin{equation}\label{equ:unique two1}AX=AA^\dag_{MN} \
\mbox{and}\ {\cal R}(X)\subseteq {\cal R}(A^\#).\end{equation} In
view of (\ref{equ:weighted star}), we know that (\ref{equ:unique
two2}) can be rewritten as
\begin{equation}\label{equ:chen y character1} A^\#AX=A^\#, \ {\cal R}(X)\subseteq
R\big(A^\#\big).\end{equation} Since $A^\#AA^\dag_{MN}=A^\#$ and
$\big(A_{MN}^\dag\big)^\#A^\#AX=(AA^\dag_{MN})^\#
AX=AA^\dag_{MN}AX=AX,$ the equivalence of (\ref{equ:unique two1})
and (\ref{equ:chen y character1}) follows.
\end{proof}

\begin{lem}{\rm (cf.\,\cite[Lemma 0.3]{chen})}\label{lem:chen-result-1}\ Let $A\in {\cal L}(H,K)$ have a closed range, and let $M\in {\cal L}(K)$ be positive definite. Then for any $X\in {\cal L}(K,H)$,
the following two statements are equivalent:
\begin{enumerate}
\item[{\rm (i)}] $AXA=A, (MAX)^*=MAX$;
\item[{\rm (ii)}] $A^*MAX=A^*M$.
\end{enumerate}
If condition {\rm (i)} is satisfied, then for any positive definite
element $N\in {\cal L}(H)$, $X$ has the form
\begin{equation}\label{equ:solution 1
3}X=A^\dag_{MN}+(I_H-A^\dag_{MN}A)Y, \ \mbox{for some}\ Y\in {\cal
L}(K,H).
\end{equation}
\end{lem}
\begin{proof} (1) Let $N$ be any positive definite element of ${\cal L}(H)$. By (\ref{equ:weighted star}) we know that conditions
(i) and (ii) can be rephrased respectively as
\begin{eqnarray}\label{eqn:1 and 3 w-m-p}&& AXA=A, \ (AX)^\#=AX,\\\label{eqn:1+3 w-m-p}&&A^\#AX=A^\#.\end{eqnarray}
Suppose that $(\ref{eqn:1 and 3 w-m-p})$ is satisfied. Then
$$A^\#AX=A^\#(AX)^\#=(AXA)^\#=A^\#.$$ Conversely, if
(\ref{eqn:1+3 w-m-p}) is satisfied, then  it is easy to show that
$(AXA-A)^\#(AXA-A)=0$, so $AXA=A$. Furthermore,
$$(AX)^\#=X^\#A^\#=X^\#A^\#AX=(A^\#AX)^\#X=(A^\#)^\#X=AX.$$

(2) Suppose that $X\in {\cal L}(K,H)$ is given such that
(\ref{eqn:1+3 w-m-p}) is satisfied. Then
$$A^\#A(X-A^\dag_{MN})=A^\#-A^\#=0\Longrightarrow\big(A(X-A^\dag_{MN})\big)^\#A(X-A^\dag_{MN})=0,$$
 so
$A(X-A^\dag_{MN})=0$; or equivalently,
$A^\dag_{MN}A(X-A^\dag_{MN})=0$, hence there exists $Y\in {\cal
L}(K,H)$ such that $X-A^\dag_{MN}=(I_H-A^\dag_{MN}A)Y$.
\end{proof}

\begin{defn}
 An element $X$ of ${\cal
L}(K,H)$ is said to be a $(1,3)$-inverse of $A\in {\cal L}(H,K)$,
written $X\in A\{1,3\}$, if $AXA=A \ \mbox{and}\  (AX)^*=AX.$
\end{defn}

\begin{prop}\label{prop:trivial-3} Let $A\in {\cal L}(H,K)$ have a closed range. Then for any $X\in
(AA^*)\{1,3\}$, we have $A^\dag=A^*X$.
\end{prop}
\begin{proof} Put $Y=A^*X$. By (\ref{equ:unique two2}) it is
sufficient to verify that
$$A^*AY=A^*, \ {\cal R}(Y)\subseteq {\cal R}(A^*).$$
The second condition is obviously satisfied. Replacing $A, M$ with
$AA^*$ and $I_K$ respectively, by ``(i)$\Longrightarrow$ (ii)" in
Lemma \ref{lem:chen-result-1} we obtain $AA^*AA^*X=AA^*$, therefore
\begin{eqnarray*}&&
A^*AY=A^*AA^*X=A^\dag (AA^*AA^*X)=A^\dag
AA^*=A^*.\qedhere\end{eqnarray*}
\end{proof}

\section{Relationship between weighted Moore-Penrose
inverses}\label{sec:relationship of w-m-p}

Throughout this section, $H$ and $K$ are two Hilbert
$\mathfrak{A}$-modules, $M\in {\cal L}(K)$ and $N_1, N_2\in {\cal
L}(H)$ are three positive definite operators. The purpose of this
section is to generalize \cite[Lemma 2.4]{xu2} from the
finite-dimensional case to the Hilbert $C^*$-module case. For any
$A\in {\cal L}(H,K)$, if ${\cal R}(A)$ is closed, then as in
\cite{xu2} we define
\begin{eqnarray}\label{equ:defn
of R M N1 N2}
R_{M;N_1,N_2}&=&I_H+(I_H-A^\dag_{MN_1}A)N_1^{-1}(N_2-N_1)\nonumber\\
&=&A^\dag_{MN_1}A+(I_H-A^\dag_{MN_1}A)N_1^{-1}N_2.\end{eqnarray}

\begin{lem}\label{lem:R} Let $A\in {\cal L}(H,K)$ have a closed range.
The operator $R_{M;N_1,N_2}$ defined by (\ref{equ:defn of R M N1
N2}) is invertible.
\end{lem}
\begin{proof} Let $P=A^\dag_{MN_1}A$, $S=(I_H-P)N_1^{-1}N_2(I_H-P)$, $H_1=(I_H-P)H$
 and $S|_{H_1}: H_1\to H_1$ be the restriction of
$S$ to $H_1$.

First, we prove that $S|_{H_1}\in {\cal L}(H_1)$ is invertible. By
the last condition in (\ref{equ:defn of WPR inverse}) we get
\begin{equation}\label{equ:the expression of N1 S} N_1S=(I_H-P)^*N_2(I_H-P)=\big(N_2^\frac
12(I_H-P)\big)^*\big(N_2^\frac 12(I_H-P)\big).\end{equation} As $P$
is idempotent, we have ${\cal N}(S)={\cal N}(N_1S)={\cal
N}(I_H-P)={\cal R}(P),$ which means that ${\cal N}(S|_{H_1})={\cal
R}(P)\cap H_1=\{0\}$. Furthermore, since ${\cal
R}\big((N_2^\frac12(I_H-P)\big)$ is closed, we may apply
Lemma~\ref{lem:orthogonal} to (\ref{equ:the expression of N1 S}) to
conclude that
\begin{eqnarray*}&& {\cal R}(S|_{H_1})={\cal R}(S)=N_1^{-1}{\cal R}(N_1S)=N_1^{-1}{\cal R}\big((I_H-P)^*N_2^\frac12\big)
\\&&={\cal R}\big(N_1^{-1}(I_H-P)^*\big)={\cal
R}\big((I_H-P)N_1^{-1}\big)={\cal R}\big(I_H-P)=H_1.\end{eqnarray*}
This completes the proof of the invertibility of $S|_{H_1}$.

Next, let
\begin{equation*}Y=P+(S|_{H_1})^{-1}(I_H-P)-(S|_{H_1})^{-1}(I_H-P)N_1^{-1}N_2P.\end{equation*}
Then since $R_{M;N_1,N_2}=P+(I_H-P)N_1^{-1}N_2P+S$, it is easy to
verify that $R_{M;N_1,N_2}Y=YR_{M;N_1,N_2}=I_H$.
\end{proof}

\begin{lem}\label{lem:result-2} {\rm (cf.\,\cite[Lemma 2.4]{xu2})}\ Suppose that $A\in {\cal L}(H,K)$ has a closed range. Then
$A^\dag_{MN_2}= R^{-1}_{M;N_1,N_2}A^\dag_{MN_1}$, where
$R_{M;N_1,N_2}$ is defined by (\ref{equ:defn of R M N1 N2}).
\end{lem}
\begin{proof}
Let $A^\#=N_1^{-1}A^*M\in {\cal L}(K_M, H_{N_1})$  be the conjugate
operator of $A\in {\cal L}(H_{N_1}, K_M)$. To simplify the notation,
we define
\begin{equation}\label{equ:defn of X}X=R_{M;N_1,N_2}^{-1}\cdot (I_H-A^\dag_{MN_1}A)N_1^{-1}N_2.\end{equation}
Then
$$(I_H-A^\dag_{MN_1}A)N_1^{-1}A^*=(I_H-A^\dag_{MN_1}A)A^\#M^{-1}=0,$$ so
by (\ref{equ:defn of X}) we have \begin{equation} \label{equ:null
space of X}XN_2^{-1}{\cal R}(A^*)=0.\end{equation} Since
$A^\dag_{MN_1}A(I_H-A^\dag_{MN_1}A)=0$, by (\ref{equ:defn of R M N1
N2}) and (\ref{equ:defn of X})  we have
\begin{eqnarray}&&X(I_H-A^\dag_{MN_1}A)=\big(R_{M;N_1,N_2}^{-1}\cdot A^\dag_{MN_1}A+X\big)(I_H-A^\dag_{MN_1}A)\nonumber\\
\label{equ:noname-2} &&=R_{M;N_1,N_2}^{-1}\cdot
\Big(A^\dag_{MN_1}A+(I_H-A^\dag_{MN_1}A)N_1^{-1}N_2\Big)(I_H-A^\dag_{MN_1}A)\nonumber\\
&&=R_{M;N_1,N_2}^{-1}\cdot R_{M;N_1,N_2}\cdot
(I_H-A^\dag_{MN_1}A)=I_H-A^\dag_{MN_1}A.\end{eqnarray} As
$I_H-A^\dag_{MN_2}A$ is the oblique projector of $H$ along
$N_2^{-1}{\cal R}(A^*)$ onto ${\cal N}(A)={\cal
R}(I_H-A^\dag_{MN_1}A)$, in view of (\ref{equ:null space of X}) and
(\ref{equ:noname-2}) we conclude that $I_H-A^\dag_{MN_2}A=X$.
Furthermore, by (\ref{equ:defn of X}) and (\ref{equ:defn of R M N1
N2}) we have
\begin{eqnarray}\label{equ:1-A+ M N2
A}&&I_H-A^\dag_{MN_2}A=X=R_{M;N_1,N_2}^{-1}\cdot(I_H-A^\dag_{MN_1}A)N_1^{-1}N_2
\\&&=R_{M;N_1,N_2}^{-1}\cdot(R_{M;N_1,N_2}-A^\dag_{MN_1}A)=I_H-R_{M;N_1,N_2}^{-1}\cdot A^\dag_{MN_1}A.\nonumber
\end{eqnarray}
It follows that
\begin{equation}\label{equ:A+ M N2 A} A^\dag_{MN_2}A=R^{-1}_{M;N_1,N_2}\cdot
A^\dag_{MN_1}A.\end{equation} Note that
$AA^\dag_{MN_1}=AA^\dag_{MN_2}$ is  the oblique projector of $K$
along $M^{-1}{\cal N}(A^*)$ onto ${\cal R}(A)$, so if we multiply
$A^\dag_{MN_1}$ from the right on both sides of (\ref{equ:A+ M N2
A}), then we may obtain
\begin{eqnarray*}\label{equ:noname-1}&&A^\dag_{MN_2}=A^\dag_{MN_2}AA^\dag_{MN_2}=A^\dag_{MN_2}AA^\dag_{MN_1}
=R^{-1}_{M;N_1,N_2}\cdot A^\dag_{MN_1}.\qedhere
\end{eqnarray*}
\end{proof}

\begin{rem}
With the notation of Lemma~\ref{lem:result-2}, by (\ref{equ:1-A+ M
N2 A}) we obtain
\begin{equation}\label{equ:R -1
equal}(I_H-A^\dag_{MN_2}A)N_2^{-1}=R_{M;N_1,N_2}^{-1}\cdot(I_H-A^\dag_{MN_1}A)N_1^{-1}.\end{equation}
\end{rem}

\section{Unified representations for weighted Moore-Penrose inverses of $1\times 2$ partitioned
operators}\label{sec:1x2}

Throughout this section, $H_1,H_2$ and $H_3$ are three Hilbert
$\mathfrak{A}$-modules, $A\in {\cal L}(H_1,H_3)$ and $B\in {\cal
L}(H_2,H_3)$ are arbitrary, $M\in {\cal L}(H_3)$ and \begin{equation}\label{equ:defn of N}N=\left(\begin{array}{ccc} N_1 & L\\
L^* & N_2\end{array}\right) \in {\cal L}(H_1\oplus
H_2)\end{equation} are two positive definite operators, where $N_1\in
{\cal L}(H_1), L\in {\cal L}(H_2,H_1)$ and $N_2\in {\cal L}(H_2)$.
By \cite[Section~5]{xu1} we know that both $N_1$ and $S(N)$  are
positive definite, where $S(N)$ is the Schur complement of $N$
defined by $$S(N)=N_2-L^*N_1^{-1}L.$$ When $A$ has a closed range, we
put
\begin{equation}\label{equ:defn of
C}C=(I_{H_3}-AA^\dag_{MN_1})\,B\in {\cal L}(H_2,H_3).\end{equation}

\begin{lem}\label{lem:condition in range} Let $A\in {\cal L}(H_1,H_3)$ have a closed range. Then
\begin{enumerate}
\item[{\rm (i)}] ${\cal R}\binom{\,A^*\,}{C^*}={\cal R}(A^*)\oplus {\cal
R}(C^*)$;

\item[{\rm (ii)}] ${\cal R}(A^*)={\cal
N}\big((I_{H_1}-A^\dag_{MN_1}A)N_1^{-1}\big)$.
\end{enumerate}
\end{lem}
\begin{proof} (i) For any $\xi, \eta\in H_3$, let
$\zeta=(AA^\dag_{MN_1})^*\xi+(I_{H_3}-AA^\dag_{MN_1})^*\eta.$ Then
 $A^*\zeta=A^*\xi$ and
$C^*\zeta=C^*\eta$, so
$\binom{A^*\xi}{C^*\eta}=\binom{\,A^*\,}{C^*}\zeta\in {\cal
R}\binom{\,A^*\,}{C^*}.$

(ii)  As $AA^\dag_{MN_1}A=A$, we have
\begin{eqnarray*}&& {\cal R}(A^*)={\cal R}\big((A^\dag_{MN_1}A)^*\big)={\cal N}
\big(I_{H_1}-(A^\dag_{MN_1}A)^*\big)\\
&&={\cal N}\big(N_1(I_{H_1}-A^\dag_{MN_1}A)N_1^{-1}\big)={\cal
N}\big((I_{H_1}-A^\dag_{MN_1}A)N_1^{-1}\big). \qedhere
\end{eqnarray*}
\end{proof}

Although the technique lemma in \cite{chen} (\cite[Lemma 0.2]{chen})
is no longer true in the infinite-dimensional case, we can still
provide a formula for $(A,C)^\dag_{MN}$ by following the line in
\cite{chen} together with some modifications.

\begin{thm}\label{thm:first main result} {\rm (cf.\,\cite[Theorem 1.1]{chen})}
Let $C$ be defined by (\ref{equ:defn of C}), and suppose that ${\cal
R}(A)$, ${\cal R}(C)$ and ${\cal R}(A,C)$  are all  closed.
 Then
\begin{equation}(A,C)^\dag_{MN}=\left(\begin{array}{ccc} A^\dag_{MN_1}-(I_{H_1}-A^\dag_{MN_1}A)N_1^{-1}LU\\
U\end{array}\right),\end{equation} where
\begin{eqnarray}&&\label{eqn:defn of
S}S=N_2-L^*(I_{H_1}-A^\dag_{MN_1}A)N_1^{-1}L=S(N)+L^*A^\dag_{MN_1}AN_1^{-1}L\in {\cal L}(H_2),\\
&&\label{eqn:defn of X2}
U=C^\dag_{MS}-(I_{H_2}-C^\dag_{MS}C)S^{-1}L^*A^\dag_{MN_1}\in {\cal
L}(H_3,H_2).
\end{eqnarray}
\end{thm}
\begin{proof}Note that $A^\dag_{MN_1}A$ is a projection on the weighted space $(H_1)_{N_1}$, so for any $\xi\in H_2$, we have
\begin{eqnarray*}&&\big<L^*A^\dag_{MN_1}AN_1^{-1}L\xi,
\xi\big>=\big<(A^\dag_{MN_1}A)(N_1^{-1}L\xi),N_1^{-1}L\xi\big>_{N_1}\ge
0,
\end{eqnarray*}
hence $L^*A^\dag_{MN_1}AN_1^{-1}L$ is positive \cite[Lemma
4.1]{lance}, which means that the operator $S$ defined by
(\ref{eqn:defn of S}) is positive definite.  Note also that
\begin{eqnarray}\label{eqn:noname-4}&&C^*MA=B^*(I_{H_3}-AA^\dag_{MN_1})^*MA=B^*M(I_{H_3}-AA^\dag_{MN_1})A=0.\hspace{2em}{}\end{eqnarray}
Now let $N_3$ be any positive definite element of ${\cal L}(H_2)$.
For any $X_1\in {\cal L}(H_3,H_1)$ and $X_2\in {\cal L}(H_3,H_2)$,
by Proposition \ref{prop:m-p inverse***} we know that
$\binom{\,X_1\,}{X_2}=(A, C)^\dag_{MN}$ if and only if
\begin{equation}\label{equ:noname-3}\binom{A^*}{C^*}M(A,C)\binom{\,X_1\,}{X_2}=\binom{A^*}{C^*}M,
\ R\left(N\binom{\,X_1\,}{X_2}\right) \subseteq
R\binom{\,A^*\,}{C^*}.\end{equation} Combining the above two
conditions with (\ref{eqn:noname-4}), we may apply
Lemma~\ref{lem:condition in range} to conclude that
$\binom{\,X_1\,}{X_2}=(A, C)^\dag_{MN}$ if and only if the following
four equations hold:
\begin{eqnarray}\label{eqn:1-}&&A^*MAX_1=A^*M;\\
\label{eqn:2-}&&C^*MCX_2=C^*M;\\
\label{eqn:3-}&&(I_{H_1}-A^\dag_{MN_1}A)N_1^{-1}(N_1X_1+LX_2)=0;\\
\label{eqn:4-}&&(I_{H_2}-C^\dag_{MN_3}C)N_3^{-1}(L^*X_1+N_2X_2)=0.
\end{eqnarray}
By (\ref{equ:solution 1 3}) we have
\begin{eqnarray}\label{eqn:5-}&&X_1=A^\dag_{MN_1}+(I_{H_1}-A^\dag_{MN_1}A)Y_1,\\
\label{eqn:6-}&&X_2=C^\dag_{MN_3}+(I_{H_2}-C^\dag_{MN_3}C)Y_2,
\end{eqnarray}
for some $Y_1\in {\cal L}(H_3,H_1)$ and $Y_2\in {\cal L}(H_3,H_2)$.
It follows from (\ref{eqn:3-}) and (\ref{eqn:5-}) that
$$(I_{H_1}-A^\dag_{MN_1}A)Y_1+(I_{H_1}-A^\dag_{MN_1}A)N_1^{-1}LX_2=0.$$
Combining the above equality with (\ref{eqn:5-}) we get
\begin{equation}\label{eqn:7-}
X_1=A^\dag_{MN_1}-(I_{H_1}-A^\dag_{MN_1}A)N_1^{-1}LX_2.\end{equation}
It follows from (\ref{eqn:4-}), (\ref{eqn:7-}) and (\ref{eqn:defn of
S}) that
\begin{equation}\label{eqn:8-}
(I_{H_2}-C^\dag_{MN_3}C)N_3^{-1}L^*A^\dag_{MN_1}+(I_{H_2}-C^\dag_{MN_3}C)N_3^{-1}SX_2=0.\end{equation}
By (\ref{eqn:6-}) we have
\begin{eqnarray}\label{eqn:9-}&&(I_{H_2}-C^\dag_{MN_3}C)N_3^{-1}SX_2=(I_{H_2}-C^\dag_{MN_3}C)N_3^{-1}SC^\dag_{MN_3}\nonumber\\
&&\hspace{12em}+(I_{H_2}-C^\dag_{MN_3}C)N_3^{-1}S(I_{H_2}-C^\dag_{MN_3}C)Y_2.\end{eqnarray}
So, if we let $N_3=S$, then by the above equality we get
\begin{equation}\label{eqn:10-}(I_{H_2}-C^\dag_{MS}C)X_2=(I_{H_2}-C^\dag_{MS}C)Y_2.\end{equation}
The expression for $U$ given by (\ref{eqn:defn of X2}) follows from
(\ref{eqn:6-}), (\ref{eqn:10-}) and (\ref{eqn:8-}) by letting
$N_3=S$. The conclusion then follows from (\ref{eqn:7-}).
\end{proof}

Theorem \ref{thm:particular representation} below was proved in
\cite{chen,miao1,wang2} for matrices by using different methods. In
the context of Hilbert $C^*$-module operators, we can give a general
proof as follows:
\begin{thm} \label{thm:particular representation} Under the conditions of Theorem \ref{thm:first main result} we have
\begin{equation}\label{equ:new expression of X1 and X2}(A,B)^\dag_{MN}=\left(\begin{array}{ccc} A^\dag_{MN_1}-\Big(D+(I_{H_1}-A^\dag_{MN_1}A)N_1^{-1}L\Big)\widetilde{U}\\
\widetilde{U}\end{array}\right),\end{equation} where
\begin{eqnarray}&&\label{eqn:defn of
D}D=A^\dag_{MN_1}B\in {\cal L}(H_2,H_1),\\
&&\label{eqn:defn of S
tilde}\widetilde{S}=N_2-L^*(I_{H_1}-A^\dag_{MN_1}A)N_1^{-1}L+D^*N_1D-D^*L-L^*D\in
{\cal L}(H_2),\\&& \label{eqn:defn of X2-}
\widetilde{U}=C^\dag_{M\widetilde{S}}+(I_{H_2}-C^\dag_{M\widetilde{S}}C)(\widetilde{S})^{-1}(D^*N_1-L^*)A^\dag_{MN_1}\in
{\cal L}(H_3,H_2).\hspace{2.5em}{}
\end{eqnarray}
\end{thm}
\begin{proof} Let $T=\Big(\begin{array}{ccc} I_{H_1} & -D\\
0 & I_{H_2}\end{array}\Big)\in {\cal L}(H_1\oplus H_2)$. Then $T$ is invertible with $T^{-1}=\Big(\begin{array}{ccc} I_{H_1} & D\\
0 & I_{H_2}\end{array}\Big)$. In view of (\ref{equ:defn of C}) and
(\ref{eqn:defn of D}), we have
\begin{equation}\label{equ:ab to ac}(A,B)\,T=(A,C),\end{equation}
which means that ${\cal R}(A,B)={\cal R}(A,C)$ is closed, so
$(A,B)^\dag_{MN}$ exists. Furthermore, by (\ref{equ:unique two2})
and (\ref{equ:ab to ac}) we know that $(A,B)^\dag_{MN}$ is the
unique solution
$\widetilde{X}=\binom{\widetilde{X_1}}{\widetilde{X_2}}\in {\cal
L}(H_3, H_1\oplus H_2)$ to the equation
\begin{eqnarray}&&(A,C)^*M(A,C)T^{-1}\widetilde{X}=(A,C)^*M,\\
&&R\big(T^*NT\cdot T^{-1}\widetilde{X}\big)\subseteq
R\Big((A,C)^*\Big).\end{eqnarray} It follows from (\ref{equ:unique
two2}) that $T^{-1}\widetilde{X}=(A,C)^\dag_{M\widetilde{N}}$, where
\begin{equation}\widetilde{N}=T^*NT=\left(\begin{array}{ccc} N_1 & L-N_1D\\
L^*-D^*N_1 & N_2-D^*L-L^*D+D^*N_1D\end{array}\right).\end{equation}
By the definition of $D$ we get $(I_{H_1}-A^\dag_{MN_1} A)D=0$ and
\begin{eqnarray*}&&D^*N_1(I_{H_1}-A^\dag_{MN_1}A)N_1^{-1}=
B^*(A^\dag_{MN_1})^*N_1(I_{H_1}-A^\dag_{MN_1}A)N_1^{-1}\\
&&=B^*(A^\dag_{MN_1})^*(I_{H_1}-A^\dag_{MN_1}A)^*=0.\end{eqnarray*}
In view of (\ref{eqn:defn of S}), if we replace $N_2, L$ with
$N_2-D^*L-L^*D+D^*N_1D$ and $L-N_1D$ respectively, and define
\begin{eqnarray*}
\widetilde{S}&=&(N_2-D^*L-L^*D+D^*N_1D)-(L-N_1D)^*(I-A^\dag_{MN_1}A)N_1^{-1}(L-N_1D)\\
&=&N_2-L^*(I_{H_1}-A^\dag_{MN_1}A)N_1^{-1}L+D^*N_1D-D^*L-L^*D,\end{eqnarray*}
then by Theorem \ref{thm:first main result} we conclude that
$T^{-1}\widetilde{X}=\binom{\,V_1\,}{V_2}$ with
\begin{eqnarray}\label{eqn:new expression of X2}V_2&=&C^\dag_{M\widetilde{S}}-(I_{H_2}-C^\dag_{M\widetilde{S}}C)(\widetilde{S})^{-1}(L-N_1D)^*A^\dag_{MN_1},\\
\label{eqn:new expression of
X1}V_1&=&A^\dag_{MN_1}-(I_{H_1}-A^\dag_{MN_1}A)N_1^{-1}(L-N_1D)V_2\nonumber\\
&=&A^\dag_{MN_1}-(I_{H_1}-A^\dag_{MN_1}A)N_1^{-1}LV_2.
\end{eqnarray}
As
$\binom{\widetilde{X_1}}{\widetilde{X_2}}=T\binom{\,V_1\,}{V_2}=\binom{V_1-DV_2}{V_2}$,
(\ref{eqn:defn of X2-}) and (\ref{equ:new expression of X1 and X2})
then follow from (\ref{eqn:new expression of X2}) and (\ref{eqn:new
expression of X1}).
\end{proof}

Now we are ready to  give a unified representation for
$(A,B)^\dag_{MN}$ in terms of $C^\dag_{MN_3}$, where $N_3\in {\cal
L}(H_2)$ can be an arbitrary positive definite operator.

\begin{thm}\label{thm:unified representation}\  Under the conditions of Theorem \ref{thm:first main result} we have
\begin{equation}\label{equ:w-m-p-formula2-c} (A,B)_{MN}^\dag
=\left(\begin{array}{c}
A_{MN_1}^\dag-\big(D+(I_{H_1}-A^\dag_{MN_1}A) N_1^{-1}L\big)V
\\ V
\end{array}\right), \end{equation} where $N_3\in {\cal L}(H_2)$ is arbitrary positive definite, $D$ and $\widetilde{S}$ are defined by
(\ref{eqn:defn of D}) and (\ref{eqn:defn of S tilde}) respectively,
and
\begin{eqnarray}\label{equ:defn of R widetilde K and
K}
&&R_{M;N_3,\widetilde{S}}=I_{H_2}+(I_{H_2}-C^\dag_{MN_3}C)N_3^{-1}(\widetilde{S}-N_3),\nonumber\\
\label{equ:defn
of V}&& V=R_{M;N_3,\widetilde{S}}^{-1}
\Big(C^\dag_{MN_3}+(I_{H_2}-C^\dag_{MN_3}C)N_3^{-1}(D^*N_1-L^*)A^\dag_{MN_1}\Big).
\end{eqnarray}
\end{thm}
\begin{proof} By Lemma \ref{lem:result-2} we have
$C_{M\widetilde{S}}^\dag=R_{M;N_3,\widetilde{S}}^{-1}\,C_{MN_3}^\dag$.
Furthermore, by (\ref{equ:R -1 equal}) we can get
$$(I_{H_2}-C^\dag_{M\widetilde{S}}C)\widetilde{S}^{-1}=R_{M;N_3,\widetilde{S}}^{-1}\cdot(I_{H_2}-C^\dag_{MN_3}C)N_3^{-1}.$$
The conclusion then follows from (\ref{equ:new expression of X1 and
X2}) and (\ref{eqn:defn of X2-}).
\end{proof}

In the special case of the preceding theorem where $N_3=S(N)$, we
regain the main technique result of \cite{xu1} as follows:

\begin{thm} {\rm (cf.\,\cite[Theorem 5.1]{xu1})}\label{thm:w-m-p formula} \  Under the conditions of Theorem \ref{thm:first main result} we have
\begin{equation}\label{equ:w-m-p-formula2} (A,B)_{MN}^\dag=\left(\begin{array}{c} A_{MN_1}^\dag-\big(\Sigma+N_1^{-1}L
\big)\Omega
\\ \Omega
\end{array}\right), \end{equation} where $C$ and $D$ are defined by (\ref{equ:defn of
C}) and (\ref{eqn:defn of D}) respectively, and
\begin{eqnarray}&& \label{eqn:defn of Sigma}\Sigma=A_{MN_1}^\dag
(B-AN_1^{-1}L)=D-A^\dag_{MN_1}AN_1^{-1}L,\\
&&\label{eqn:defn of Y} Y=\big(I- C_{MS(N)}^\dag C \big) S(N)^{-1},
\\ \label{eqn:defn of Omega}&&\label{eqn:defn of Omega} \Omega=\big(I+Y\Sigma^* N_1 \Sigma
\big)^{-1}\Big(Y\Sigma^* N_1\cdot A_{MN_1}^\dag+C_{MS(N)}^\dag\Big).
\end{eqnarray}
\end{thm}
\begin{proof} Let $\widetilde{S}$ be given  by (\ref{eqn:defn of S
tilde}) and define \begin{equation}\label{equ:defn of
Delta}\Delta=\widetilde{S}-S(N)=L^*A^\dag_{MN_1}AN_1^{-1}L+D^*N_1D-D^*L-L^*D.\end{equation}
By definition  we have
\begin{equation}\label{equ:expression of Sigma*}\Sigma^*=D^*-L^*A^\dag_{MN_1}AN_1^{-1},
\ \mbox{so}\ \Sigma^*N_1=D^*N_1-L^*A^\dag_{MN_1}A.\end{equation} It
follows that
$\Sigma^*N_1A^\dag_{MN_1}=D^*N_1A^\dag_{MN_1}-L^*A^\dag_{MN_1}.$
Therefore,
\begin{equation}\label{equ:S(N)-technique-1}
(I-C^\dag_{MS(N)}C)S(N)^{-1}(D^*N_1-L^*)A^\dag_{MN_1}=Y\Sigma^*N_1A^\dag_{MN_1}.
\end{equation}

By the definition of $D$, we have $A^\dag_{MN_1}AD=D$, so by
(\ref{equ:expression of Sigma*}) and (\ref{eqn:defn of Sigma}) we
have
\begin{eqnarray}&&\Sigma^*N_1\Sigma=(D^*N_1-L^*A^\dag_{MN_1}A)(D-A^\dag_{MN_1}AN_1^{-1}L)\nonumber\\
&&=D^*N_1D-D^*(A^\dag_{MN_1}A)^*L-L^*A^\dag_{MN_1}AD+L^*A^\dag_{MN_1}AN_1^{-1}L\nonumber\\
\label{eqn:equal to
Delta}&&=D^*N_1D-D^*L-L^*D+L^*A^\dag_{MN_1}AN_1^{-1}L=\Delta.
\end{eqnarray}It follows that
\begin{equation}\label{equ:S(N)-technique-2}R_{M;S(N),\widetilde{S}}=I+Y\Delta=I+Y\Sigma^*N_1\Sigma.\end{equation}
Finally, by the definitions of $D$ and $\Sigma$ we get
\begin{equation}\label{equ:S(N)-technique-3}D+(I-A^\dag_{MN_1}A) N_1^{-1}L=\Sigma+N_1^{-1}L.\end{equation}
Expression (\ref{eqn:defn of Omega}) for $\Omega$ follows from
(\ref{equ:defn of V}), (\ref{equ:S(N)-technique-2}) and
(\ref{equ:S(N)-technique-1}). Formula (\ref{equ:w-m-p-formula2}) for
$(A,B)_{MN}^\dag$ then follows from (\ref{equ:w-m-p-formula2-c}) and
(\ref{eqn:defn of Sigma}).
\end{proof}

\section{Representations for weighted Moore-Penrose inverses of $2\times 2$ partitioned
operators}\label{sec:weighted MP-inverse of 2x2 matrices}

\subsection{Non weighted case}


Following the line initiated in \cite{miao2}, in this section we
study the representations for the (non-weighted) Moore-Penrose inverse
$A^\dag$ of a general $2\times 2$ partitioned operator matrix
\begin{equation}\label{equ:defn of A}A=\left(
                             \begin{array}{cc}
                               A_{11} & A_{12} \\
                               A_{21} & A_{22} \\
                             \end{array}
                           \right)
\in {\cal L}(H_1\oplus H_2, K_1\oplus K_2),\end{equation} where
$H_1,H_2, K_1$ and $K_2$ are four Hilbert $\mathfrak{A}$-modules,
$A_{11}\in {\cal L}(H_1,K_1)$, $A_{12}\in {\cal L}(H_2, K_1)$,
$A_{21}\in {\cal L}(H_1,K_2)$ and $A_{22}\in {\cal L}(H_2,K_2)$. In
the case when $A_{11}$ has a closed range, let $S(A)$ be the Schur complement
of $A$ defined by
\begin{equation}\label{eqn:defn of S(A)}
S(A)=A_{22}-A_{21}A_{11}^\dag A_{12}\in {\cal
L}(H_2,K_2).\end{equation}

\subsubsection{Special case}

\begin{lem}\label{lem:F_i(A)+} Suppose that $A_{11}$ has a closed range. Then
both $F_1(A)^\dag $ and  $F_2(A)^\dag $ exist, where
\begin{eqnarray}&&\label{eqn:defn of F_1(A)}F_1(A)=\binom{-A_{11}^\dag A_{12}}{I_{H_2}}\in {\cal L}(H_2, H_1\oplus H_2), \\
&&\label{eqn:defn of F_2(A)}F_2(A)=(-A_{21}A_{11}^\dag, I_{K_2})\in
{\cal L}(K_1\oplus K_2, K_2).\end{eqnarray} Furthermore, the
following equalities hold:
\begin{enumerate}
\item[{\rm (i)}] $F_1(A)^\dag \cdot \left(
                                                             \begin{array}{cc}
                                                               A_{11}^\dag A_{11} & A_{11}^\dag A_{12}\\
                                                               0 & 0 \\
                                                             \end{array}
                                                           \right)=F_1(A)^\dag-(0,I_{H_2})$;
\item[{\rm (ii)}] $\left(\begin{array}{cc}
                      A_{11}A_{11}^\dag & 0 \\
                      A_{21}A_{11}^\dag & 0 \\
                    \end{array}
                  \right)\cdot
F_2(A)^\dag=F_2(A)^\dag-\left(
                          \begin{array}{c}
                            0 \\
                            I_{K_2} \\
                          \end{array}
                        \right)$.
\end{enumerate}
\end{lem}
\begin{proof} By definition we have $F_2(A)F_2(A)^*=I_{K_2}+\big(A_{21}A_{11}^\dag\big)\,\big(A_{21}A_{11}^\dag\big)^*$, which is invertible,
hence by Proposition~\ref{prop:trivial-3} we have
\begin{equation}\label{equ:S_2(A)+}F_2(A)^\dag=F_2(A)^*\cdot\big(F_2(A)F_2(A)^*\big)^{-1}.\end{equation} It follows from
(\ref{eqn:defn of F_2(A)}) and (\ref{equ:S_2(A)+}) that
\begin{eqnarray*}&&\left(
                    \begin{array}{cc}
                      A_{11}A_{11}^\dag & 0 \\
                      A_{21}A_{11}^\dag & 0 \\
                    \end{array}
                  \right)\,F_2(A)^\dag
=\left(
     \begin{array}{cc}
       -(A_{21}A_{11}^\dag)^* \\
       -(A_{21}A_{11}^\dag)(A_{21}A_{11}^\dag)^* \\
     \end{array}
   \right)\big(F_2(A)F_2(A)^*\big)^{-1}\\
   &&=\left[F_2(A)^*-\left(
                                                       \begin{array}{c}
                                                         0 \\
                                                         F_2(A)F_2(A)^* \\
                                                       \end{array}
                                                     \right)\right]\big(F_2(A)F_2(A)^*\big)^{-1}=F_2(A)^\dag-\left(
                          \begin{array}{c}
                            0 \\
                            I_{K_2} \\
                          \end{array}
                        \right).
\end{eqnarray*}
The proof of (i) is similar.
\end{proof}

\begin{thm}\label{thm:2 x 2 first result} {\rm (cf.\,\cite[Theorem 2]{miao2})}\ Suppose
that  both $A_{11}$ and $S(A)$ have closed ranges, and
\begin{equation}\label{equ:condition of two ranges}
(I_{K_1}-A_{11}A_{11}^\dag)A_{12}=0,\  A_{21}(I_{H_1}-A_{11}^\dag
A_{11})=0.\end{equation} Then
\begin{equation}\label{equ:expression of A+-special case}A^\dag
=X_L(A)\,{\rm
diag}(A_{11}^\dag,0)\,X_R(A)+F_1(A)\,S(A)^g\,F_2(A),\end{equation}
where $F_1(A)$ and $F_2(A)$ are defined  by (\ref{eqn:defn of
F_1(A)}) and (\ref{eqn:defn of F_2(A)}) respectively, and
\begin{eqnarray}
&&\label{eqn:defn of
S(A)^g}S(A)^g=S(A)^\dag_{\left[F_2(A)F_2(A)^*\right]^{-1},F_1(A)^*F_1(A)}\in
{\cal L}(K_2,H_2),\\
&&\label{eqn:defn of X_L(A)}X_L(A)=I_{H_1\oplus H_2}-F_1(A)
\big[I_{H_2}-S(A)^g S(A)\big]F_1(A)^\dag\in {\cal
L}(H_1\oplus H_2),\\
&&\label{eqn:defn of X_R}X_R(A)=I_{K_1\oplus K_2}-F_2(A)^\dag
\big[I_{K_2}-S(A)S(A)^g\big]F_2(A)\in {\cal L}(K_1\oplus K_2).
\end{eqnarray}
\end{thm}
\begin{proof} It follows from (\ref{equ:defn of A}),  (\ref{eqn:defn of F_1(A)}), (\ref{eqn:defn of F_2(A)}) and
(\ref{equ:condition of two ranges}) that
\begin{equation}\label{equ:expression of AX}AF_1(A)=\binom{0}{S(A)}\ \mbox{and}\ \ F_2(A)A=\big(0,S(A)\big),\end{equation}
which implies that
\begin{equation}\label{equ:two items =A}AX_L(A)=A \ \mbox{and}\ X_R(A)A=A.\end{equation}
To simplify the notation, let
\begin{eqnarray}&&\lambda_1(A)=I_{H_2}-S(A)^gS(A)\
\mbox{and}\ \lambda_2(A)=I_{K_2}-S(A)S(A)^g.\end{eqnarray} Then by
(ii) of Lemma \ref{lem:F_i(A)+} we have
\begin{eqnarray}\begin{split}&A\,{\rm diag}(A_{11}^\dag,0)\,X_R(A)=\left(
     \begin{array}{cc}
       A_{11}A_{11}^\dag & 0 \\
       A_{21}A_{11}^\dag & 0 \\
     \end{array}
   \right)X_R(A)\\
&=\left(
     \begin{array}{cc}
       A_{11}A_{11}^\dag & 0 \\
       A_{21}A_{11}^\dag & 0 \\
     \end{array}
   \right)-\left(
     \begin{array}{cc}
       A_{11}A_{11}^\dag & 0 \\
       A_{21}A_{11}^\dag & 0 \\
     \end{array}
   \right)F_2(A)^\dag \lambda_2(A)F_2(A)\\
   &\label{equ:first computation of AA+}=\left(
     \begin{array}{cc}
       A_{11}A_{11}^\dag & 0 \\
       A_{21}A_{11}^\dag & 0 \\
     \end{array}
   \right)-F_2(A)^\dag\lambda_2(A)F_2(A)+\binom{0}{I_{K_2}}\lambda_2(A)F_2(A).\end{split}
\end{eqnarray}
Furthermore, by the first equality in (\ref{equ:expression of AX})
we get
\begin{eqnarray}\label{equ:second computation of AA+}\begin{split}&AF_1(A)S(A)^gF_2(A)=\binom{0}{I_{K_2}}S(A)S(A)^gF_2(A)\\
&=\binom{0}{I_{K_2}}\big(I_{K_2}-\lambda_2(A)\big)F_2(A)=\binom{0}{I_{K_2}}F_2(A)-\binom{0}{I_{K_2}}\lambda_2(A)F_2(A)\\
&=\left(\begin{array}{cc} 0 & 0
\\-A_{21}A_{11}^\dag & I_{K_2}\\ \end{array}
\right)-\binom{0}{I_{K_2}}\lambda_2(A)F_2(A).\end{split}
\end{eqnarray}

Now let $Z$ be the right side of (\ref{equ:expression of A+-special
case}). Then by the first equality in (\ref{equ:two items =A}),
(\ref{equ:first computation of AA+}) and (\ref{equ:second
computation of AA+}), we get
\begin{eqnarray}\label{eqn:AZ equals what}&&AZ={\rm diag}(A_{11}A_{11}^\dag,
   I_{K_2})-F_2(A)^*\big(F_2(A)F_2(A)^*\big)^{-1}\lambda_2(A)F_2(A),
\end{eqnarray}
which means that $(AZ)^*=AZ$, since by the definitions of $S(A)^g$
and $\lambda_2(A)$ we have
$$\lambda_2(A)^*=\left(F_2(A)F_2(A)^*\right)^{-1}
\lambda_2(A)\big(F_2(A) F_2(A)^*\big).$$ As
$\lambda_2(A)\big(0,S(A)\big)=0$, we may combine (\ref{eqn:AZ equals
what}) with the second equality in (\ref{equ:expression of AX})  to
get
$$AZA={\rm diag}(A_{11}A_{11}^\dag, I_{K_2})A=A.$$ Similarly, as
$F_1(A)^\dag=\big(F_1(A)^*F_1(A)\big)^{-1}F_1(A)^*$ and $${\rm
diag}(A_{11}^\dag A_{11},I_{H_2}\big)X_L(A)=X_L(A)-{\rm
diag}\big(I_{H_1}-A_{11}^\dag A_{11},0\big),$$ we can prove that
$$ZA={\rm diag}(A_{11}^\dag A_{11},
I_{H_2})-F_1(A)\lambda_1(A)\big(F_1(A)^*F_1(A)\big)^{-1}F_1(A)^*$$
with $(ZA)^*=ZA$ and $ZAZ=Z$, therefore $Z=A^\dag$.
\end{proof}

\begin{cor}\label{cor:representation MP-inverse of positive matrix} Let
$A=\Big(\begin{array}{ccc} A_{11} & A_{12}\\
A_{12}^* & A_{22}\end{array}\Big)\in {\cal L}(K_1\oplus K_2)$ be
positive, where $A_{ij}\in {\cal L}(K_j,K_i)$$(i,j=1,2)$. If both
${\cal R}(A_{11})$ and ${\cal R}\big(S(A)\big)$ are closed, then
\begin{equation}\label{equ:expression of A+-special case-}A^\dag=X_L(A)\,{\rm diag}(A_{11}^\dag,0)\,X_R(A)+F_1(A)\,S(A)^g\,F_1(A)^*,\end{equation}
where $F_1(A)$ is defined by (\ref{eqn:defn of F_1(A)}),  $S(A)^g$,
$X_L(A)$ and $X_R(A)$ are given respectively as (\ref{eqn:defn of
S(A)^g}), (\ref{eqn:defn of X_L(A)}) and (\ref{eqn:defn of X_R}) by
letting $F_2(A)$ be replaced with $F_1(A)^*$. In addition, a
$\{1,3\}$-inverse of $A$ can be given by
\begin{equation}\label{equ:1 3 inverse of A} A^{(1,3)}={\rm diag}(A^\dag_{11},0)\,X_R(A)+F_1(A)\,S(A)^g\,F_1(A)^*.
\end{equation}
\end{cor}
\begin{proof}Since $A$ is positive, by
\cite[Corollary 3.5]{xu3} we have
\begin{equation}\label{equ:three conditions of positive matrix}A_{11}\ge 0, \ A_{12}=A_{11}A_{11}^\dag A_{12} \  \mbox{and}\ S(A)\ge 0.
\end{equation}
As $(A_{11}^\dag)^*=A_{11}^\dag$, conditions in (\ref{equ:condition
of two ranges}) are satisfied. Note that in this case
$F_2(A)=\big(F_1(A)\big)^*$, (\ref{equ:expression of A+-special
case-}) follows from (\ref{equ:expression of A+-special case}). Let
$A^{(1,3)}$ be the operator given by (\ref{equ:1 3 inverse of A}).
As $AX_L(A)=A$ we have $AA^{(1,3)}=AA^\dag$, so $A^{(1,3)}$ is a
$\{1,3\}$-inverse of $A$.
\end{proof}

\subsubsection{General case} Let
\begin{equation}\label{equ:defn of AA*=E}E=AA^*\stackrel{def}{=}\left(
            \begin{array}{cc}
              E_{11} & E_{12} \\
              E_{12}^* & E_{22} \\
            \end{array}
          \right)\in {\cal L}(K_1\oplus K_2).\end{equation}
If $(A_{11}, A_{12})$ has a closed range, then as $E_{11}=(A_{11},
A_{12})(A_{11}, A_{12})^*$, by  Lemma~1.1 and
Proposition~\ref{prop:trivial-3} we know that $E_{11}^\dag$ exists
such that $(A_{11},A_{12})^\dag=(A_{11},A_{12})^* E_{11}^\dag$. Let
$S(E)=E_{22}-E_{12}^*E_{11}^\dag E_{12}$ be the Schur complement of
$E$. Assuming further that both $A$ and $S(E)$ have  closed ranges,
then for any $\{1,3\}$-inverse $E^{(1,3)}$ of $E$, we have
$A^\dag=A^*E^{(1,3)}$. In particular,  by (\ref{equ:1 3 inverse of
A}) we have
\begin{equation}\label{equ:formula for A dag-general} A^\dag=A^*\cdot \left[{\rm
diag}(E^\dag_{11},0)\,X_R(E)+F_1(E)\,S(E)^g\,F_1(E)^*\right],
\end{equation}
where
\begin{eqnarray}&&\label{eqn:defn of F_1(E)}F_1(E)=\binom{-E_{11}^\dag E_{12}}{I_{K_2}}\in {\cal L}(K_2, K_1\oplus K_2), \\
&&\label{eqn:S-E-g}S(E)^g=S(E)^\dag_{\left[F_1(E)^*F_1(E)\right]^{-1},F_1(E)^*F_1(E)}\in
{\cal L}(K_2),\\
&&\label{eqn:X-R_E}X_R(E)=I_{K_1\oplus K_2}-\big(F_1(E)^*\big)^\dag
\big(I_{K_2}-S(E)S(E)^g\big)F_1(E)^*\in {\cal L}(K_1\oplus K_2).
\end{eqnarray}

\subsection{The weighted case}\label{sec:weighted case}

Following the line initiated in \cite{xu1} for $1\times 2$
partitioned operators, in this subsection we provide an approach
 to
the construction of  Moore-Penrose inverses of $2\times 2$ partitioned
operators from the non-weighted case to the weighted case. A
detailed description of our idea can be illustrated as follows.

For any Hilbert $\mathfrak{A}$-module $X$, and any projection $P$ of
${\cal L}(X)$, let $X_1=PX$ and $X_2=(I_X-P)X$, and define
$\lambda_X: X\to X_1\oplus X_2$ by
\begin{equation}\label{equ:lambda x}\lambda_X(x)=\binom{Px}{x-Px},\ \mbox{for any}\ \ x\in
X.\end{equation} Then $\lambda_X$ is a unitary operator with
$\lambda_X^*=\lambda_X^{-1}$, where $\lambda_X^{-1}: X_1\oplus
X_2\to X$ is given by
$$\lambda_X^{-1}\binom{x_1}{x_2}=x_1+x_2,\ \mbox{for any}\ x_i\in X_i,
i=1,2.$$

Now let $H_1$ and $H_2$ be two Hilbert
$\mathfrak{A}$-modules, \begin{equation}\label{equ:2 by 2 of N}N=\left(\begin{array}{ccc} N_{11} & N_{12}\\
N_{12}^* & N_{22}\end{array}\right)\in {\cal L}(H_1\oplus
H_2)\end{equation} be a positive definite operator, where $N_{11}\in
{\cal L}(H_1)$, $N_{12}\in {\cal L}(H_2, H_1)$ and $N_{22}\in {\cal
L}(H_2)$. Let $S(N)=N_{22}-N_{12}^*N_{11}^{-1}N_{12}$ be the Schur
complement of $N$. Define
\begin{equation}\label{equ:defn of P a X} a=N_{11}^{-1}N_{12}, \ P=\left(\begin{array}{ccc} I_{H_1} & a\\
0 & 0\end{array}\right) \ \mbox{and}\quad X=(H_1\oplus
H_2)_N.\end{equation} Then $P^2=P$ and $NP=P^*N$, so
$P^\#=N^{-1}P^*N=P$, which means that $P\in {\cal L}(X)$ is a
projection of ${\cal L}(X)$, where $X$ is the weighted space define
by (\ref{equ:defn of P a X}) whose inner-product is given by
\begin{eqnarray*} &&\Big<\binom{x_1}{y_1}, \binom{x_2}{y_2}\Big>_N=\Big<\binom{x_1}{y_1}, N\binom{x_2}{y_2}\Big>=\big<x_1,
N_{11}x_2+N_{12}y_2\big>+\big<y_1,
N_{12}^*x_2+N_{22}y_2\big>\end{eqnarray*} for any $x_i\in H_1$ and
$y_i\in H_2, i=1,2.$ By (\ref{equ:defn of P a X}) we have
\begin{eqnarray}\label{eqn:defn of X1} && X_1=PX=\left\{\binom{h_1+ah_2}{0}\,\bigg|\,h_i\in H_i\right\}=\left\{
\bigg(\begin{array}{c} u\\0\end{array}\bigg)\,\bigg|\,u\in H_1
\right\},\hspace{1em}{}\\
\label{eqn:defn of X2+}&&
X_2=(I_X-P)X=\left\{\binom{-ah_2}{h_2}\,\bigg|\, h_2\in
H_2\right\}.\end{eqnarray} With the inner products inherited from
$X$,  both $X_1$ and $X_2$ are Hilbert $\mathfrak{A}$-modules. Let
$j_{H_1}: (H_1)_{N_{11}}\to X_1$ and $j_{H_2}: (H_2)_{S(N)}\to X_2$
be defined by
$$j_{H_1}(h_1)=\binom{h_1}{0}\ \mbox{and} \ j_{H_2}(h_2)=\binom{-ah_2}{h_2},\
\mbox{for any}\ h_i\in H_i, i=1,2.$$ It is easy to verify that
both $j_{H_1}$ and $j_{H_2}$  are unitary operators with
$$j_{H_1}^{-1}\binom{h_1}{0}=h_1\ \mbox{and}\
j_{H_2}^{-1}\binom{-ah_2}{h_2}=h_2,\ \mbox{for any}\ h_i\in H_i,
i=1,2.$$ Let $j_{H_1}\oplus j_{H_2}: (H_1)_{N_{11}}\oplus
(H_2)_{S(N)}\to X_1\oplus X_2$ be the associated unitary operator
defined by
$$(j_{H_1}\oplus
j_{H_2})\binom{h_1}{h_2}=\binom{j_{H_1}(h_1)}{j_{H_2}(h_2)}=\left(
                                                              \begin{array}{c}
                                                                h_1 \\
                                                                0 \\
                                                                ---\\
                                                                -ah_2 \\
                                                                h_2 \\
                                                              \end{array}
                                                            \right)
,\ \mbox{for any}\ h_i\in H_i, i=1,2.$$ Then clearly,
$(j_{H_1}\oplus j_{H_2})^\#=(j_{H_1}\oplus
j_{H_2})^{-1}=j_{H_1}^{-1}\oplus j_{H_2}^{-1}=j_{H_1}^\#\oplus
j_{H_2}^\#.$

Now suppose that $K_1$ and
$K_2$ are two additional Hilbert $\mathfrak{A}$-modules, and \begin{equation}\label{equ:2 by 2 of M}M=\left(\begin{array}{ccc} M_{11} & M_{12}\\
M_{12}^* & M_{22}\end{array}\right)\in {\cal L}(K_1\oplus
K_2)\end{equation} is a positive definite operator, where $M_{11}\in
{\cal L}(K_1)$, $M_{12}\in {\cal L}(K_2, K_1)$ and $M_{22}\in {\cal
L}(K_2)$. Let $S(M)=M_{22}-M_{12}^*M_{11}^{-1}M_{12}$ be the Schur
complement of $M$, and define
\begin{equation}\label{equ:defn of Q b Y} b=M_{11}^{-1}M_{12}, \ Q=\left(\begin{array}{ccc} I_{K_1} & b\\
0 & 0\end{array}\right) \ \mbox{and}\quad Y=(K_1\oplus
K_2)_M.\end{equation} Similarly, define $Y_1=QY, Y_2=(I_Y-Q)Y,
\lambda_Y: Y\to Y_1\oplus Y_2$, $j_{K_1}: (K_1)_{M_{11}}\to Y_1$ and
$j_{K_2}: (K_2)_{S(M)}\to Y_2$.

With the notation as above and suppose further that $$A=\left(
                                             \begin{array}{cc}
                                               A_{11} & A_{12} \\
                                               A_{21} & A_{22} \\
                                             \end{array}
                                           \right)\in {\cal
                                           L}(H_1\oplus H_2,
                                           K_1\oplus K_2),$$ where
$A_{11}\in {\cal L}(H_1,K_1), A_{12}\in {\cal L}(H_2, K_1),
A_{21}\in {\cal L}(H_1,K_2)$ and $A_{22}\in {\cal L}(H_2,K_2)$. Then
we have the following commutative diagram:

\centerline{\xymatrix{\label{commutative diagram}
  (H_1)_{N_{11}}\oplus (H_2)_{S(N)} \ar[d]_{B} \ar[r]^-{j_{H_1}\oplus j_{H_2}} & X_1\oplus X_2 \ar[r]^-{\lambda_X^{-1}} &  X=(H_1\oplus H_2)_N \ar[d]^{A} \\
 (K_1)_{M_{11}}\oplus (K_2)_{S(M)} \ar[r]^-{j_{K_1}\oplus j_{K_2}} & Y_1\oplus Y_2\ar[r]^-{\lambda_Y^{-1}} & Y=(K_1\oplus K_2)_M }}
 \noindent where \begin{equation}\label{equ:defn of B++}B=\left(
           \begin{array}{cc}
             B_{11} & B_{12} \\
             B_{21} & B_{22} \\
           \end{array}
         \right)=(j_{K_1}^{-1}\oplus j_{K_2}^{-1})\circ
         \lambda_Y\circ A\circ \lambda_X^{-1}\circ (j_{H_1}\oplus
         j_{H_2}),\end{equation}
         with
         \begin{eqnarray} &&\label{eqn:defn of B11}
         B_{11}=A_{11}+M_{11}^{-1}M_{12}A_{21},\\&&\label{eqn:defn of B12}B_{12}=A_{12}+M_{11}^{-1}M_{12}A_{22}-A_{11}N_{11}^{-1}N_{12}-M_{11}^{-1}M_{12}A_{21}N_{11}^{-1}N_{12},\\
         &&\label{eqn:defn of B21}B_{21}=A_{21},\\
         &&\label{eqn:defn of B22}B_{22}=A_{22}-A_{21}N_{11}^{-1}N_{12}.
         \end{eqnarray}
Since $\lambda_X, \lambda_Y, j_{H_1}\oplus j_{H_2}$ and
$j_{K_1}\oplus j_{K_2}$ are all unitary operators, by (\ref{equ:defn
of B++}) we get
\begin{equation}\label{equ:abstract representation of A + M N}A^\dag_{MN}=\lambda_X^{-1}\circ (j_{H_1}\oplus
j_{H_2})\circ B^\dag_{{\rm diag}\big(M_{11},S(M)\big),{\rm
diag}\big(N_{11},S(N)\big)}\circ (j_{K_1}^{-1}\oplus
j_{K_2}^{-1})\circ \lambda_Y.\end{equation} So if we let
$$B^\dag_{{\rm diag}\big(M_{11},S(M)\big),{\rm
diag}\big(N_{11},S(N)\big)} =\left(
                                          \begin{array}{cc}
                                            (B^\dag)_{11} & (B^\dag)_{12} \\
                                            (B^\dag)_{21} & (B^\dag)_{22} \\
                                          \end{array}
                                        \right),$$
where \begin{eqnarray*}&& (B^\dag)_{11}\in {\cal
L}\big((K_1)_{M_{11}},(H_1)_{N_{11}}\big), \quad (B^\dag)_{12}\in
{\cal
L}\big((K_2)_{S(M)},(H_1)_{N_{11}}\big), \\
&& (B^\dag)_{21}\in {\cal L}\big((K_1)_{M_{11}},(H_2)_{S(N)}\big),\
(B^\dag)_{22}\in {\cal
L}\big((K_2)_{S(M)},(H_2)_{S(N)}\big),\end{eqnarray*} then by
(\ref{equ:abstract representation of A + M N}) we conclude that
$A^\dag_{MN}=\left(
                                                                   \begin{array}{cc}
                                                                     (A^\dag_{MN})_{11} & (A^\dag_{MN})_{12} \\
                                                                     (A^\dag_{MN})_{21} & (A^\dag_{MN})_{22}\\
                                                                   \end{array}
                                                                 \right)$
with $(A^\dag_{MN})_{11}\in {\cal L}(K_1,H_1)$,
$(A^\dag_{MN})_{12}\in {\cal L}(K_2,H_1)$, $(A^\dag_{MN})_{21}\in
{\cal L}(K_1,H_2)$ and $(A^\dag_{MN})_{22}\in {\cal L}(K_2,H_2)$,
such that
\begin{eqnarray} &&\label{eqn:A11}
         (A^\dag_{MN})_{11}=(B^\dag)_{11}-N_{11}^{-1}N_{12}\,(B^\dag)_{21},
         \\&&\label{eqn:A12}(A^\dag_{MN})_{12}=(B^\dag)_{11}M_{11}^{-1}M_{12}+(B^\dag)_{12}
         -N_{11}^{-1}N_{12}\,(B^\dag)_{21}M_{11}^{-1}M_{12}-N_{11}^{-1}N_{12}\,(B^\dag)_{22},\hspace{2em}{}
         \\
         &&\label{eqn:A21}(A^\dag_{MN})_{21}=(B^\dag)_{21},\\
         &&\label{eqn:A22}(A^\dag_{MN})_{22}=(B^\dag)_{21}M_{11}^{-1}M_{12}+(B^\dag)_{22}.
         \end{eqnarray}
Note that $(H_1)_{N_{11}}, (H_2)_{S(N)}, (K_1)_{M_{11}}$ and
$(K_2)_{S(M)}$ are all Hilbert $\mathfrak{A}$-modules, the
Moore-Penrose inverse  of $B_{11}\in {\cal
L}\big((H_1)_{N_{11}},(K_1)_{M_{11}}\big)$ equals
$(B_{11})^\dag_{M_{11},N_{11}}$, and the adjoint operator
$B_{11}^\#$ of $B_{11}\in {\cal
L}\big((H_1)_{N_{11}},(K_1)_{M_{11}}\big)$ equals
$N_{11}^{-1}B_{11}^*M_{11}\in {\cal L}(K_1,H_1)$. Since formula
(\ref{equ:formula for A dag-general}) is valid for any Hilbert
$\mathfrak{A}$-module operators, we may use this formula to get a
concrete expression for $B^\dag_{{\rm
diag}\big(M_{11},S(M)\big),{\rm diag}\big(N_{11},S(N)\big)}$, and
then obtain an expression for $A^\dag_{MN}$ by
(\ref{eqn:A11})--(\ref{eqn:A22}).

\vspace{2ex}

\section{A numerical example}

\begin{ex}
Let $M=\left(
                    \begin{array}{cccc}
                      2 & 0 & 1 & 0 \\
                      0 & 1 & 0 & 0 \\
                      1 & 0 & 1 & 0 \\
                      0 & 0 & 0 & 1 \\
                    \end{array}
                  \right)$, $N=\left(
                                  \begin{array}{cccc}
                                    2 & 1 & 1 & 0 \\
                                    1 & 2 & 0 & 0 \\
                                    1 & 0 & 1 & 0 \\
                                    0 & 0 & 0 & 1 \\
                                  \end{array}
                                \right)$ and
                                $A=\left(
       \begin{array}{cc}
         A_{11} & A_{12} \\
         A_{21} & A_{22} \\
       \end{array}
     \right)$ with
\begin{eqnarray*}&&A_{11}=\left(
                            \begin{array}{cc}
                              1 & 0 \\
                              0 & 0 \\
                            \end{array}
                          \right), A_{12}=\left(
                                     \begin{array}{cc}
                                       1 & -1 \\
                                       1 & 3 \\
                                     \end{array}
                                   \right), A_{21}=\left(
                                              \begin{array}{cc}
                                                0 & -2 \\
                                                0 & 0 \\
                                              \end{array}
                                            \right) \ \mbox{and}\  A_{22}=\left(
                                                       \begin{array}{cc}
                                                         0 & 2 \\
                                                         0 & 0 \\
                                                       \end{array}
                                                     \right).
                                                     \end{eqnarray*}
Then $M_{11}=\left(
           \begin{array}{cc}
             2 & 0 \\
             0 & 1 \\
           \end{array}
         \right),
N_{11}=\left(
           \begin{array}{cc}
             2 & 1 \\
             1 & 2 \\
           \end{array}
         \right)$,  $S(M)=\left(
         \begin{array}{cc}
           \frac12 & 0 \\
           0 & 1 \\
         \end{array}
       \right)$ and $S(N)=\left(
                         \begin{array}{cc}
                           \frac{1}{3} & 0 \\
                           0 & 1 \\
                         \end{array}
                       \right)$. By (\ref{eqn:defn of B11})--(\ref{eqn:defn of
         B22}) we have
\begin{eqnarray*}&&B_{11}=\left(
                          \begin{array}{cc}
                            1 & -1 \\
                            0 & 0 \\
                          \end{array}
                        \right), B_{12}=\left(
                                   \begin{array}{cc}
                                     0 & 0 \\
                                     1 & 3 \\
                                   \end{array}
                                 \right), B_{21}=\left(
                                                   \begin{array}{cc}
                                                     0 & -2 \\
                                                     0 & 0 \\
                                                   \end{array}
                                                 \right), B_{22}=\left(
                                                                   \begin{array}{cc}
                                                                     -\frac23 & 2 \\
                                                                     0 & 0 \\
                                                                   \end{array}
 \right).\end{eqnarray*}
Note that the matrix $B=(B_{ij})_{1\le i,j\le 2}$, regarded as an
element of \begin{eqnarray*}&&{\cal L}\left((H_1)_{N_{11}}\oplus
(H_2)_{S(N)}, (K_1)_{M_{11}}\oplus (K_2)_{S(M)}\right)\\
&&={\cal L}\left(\left(H_1\oplus H_2\right)_{{\rm
diag}\big(N_{11},S(N)\big)},\left(K_1\oplus K_2\right)_{{\rm
diag}\big(M_{11},S(M)\big)}\right),\end{eqnarray*} whose conjugate
$B^\#$ is given by
\begin{eqnarray*}&&B^\#={\rm
diag}\big(N_{11},S(N)\big)^{-1}\cdot B^*\cdot {\rm
diag}\big(M_{11},S(M)\big)=\left(
                             \begin{array}{cccc}
                               2 & 0 & \frac{1}{3} & 0 \\
                               -2 & 0 & -\frac{2}{3} & 0 \\
                               0 & 3 & -1 & 0 \\
                               0 & 3 & 1 & 0 \\
                             \end{array}
                           \right).
\end{eqnarray*}
Let $E=BB^\#=\left(
              \begin{array}{cc}
                E_{11} & E_{12} \\
                E_{12}^* & E_{22} \\
              \end{array}
            \right)
\in {\cal L}\left((K_1)_{M_{11}}\oplus (K_2)_{S(M)}\right)$, where
\begin{eqnarray*}&&E_{11}=\left(
           \begin{array}{cc}
             4 & 0 \\
             0 & 12 \\
           \end{array}
         \right), E_{12}=\left(
                           \begin{array}{cc}
                             1 & 0 \\
                             2 & 0 \\
                           \end{array}
                         \right),
E_{21}=\left(
  \begin{array}{cc}
    4 & 4 \\
    0 & 0 \\
  \end{array}
\right)   \ \mbox{and}\ E_{22}=\left(
                                                       \begin{array}{cc}
                                                         4 & 0 \\
                                                         0 & 0 \\
                                                       \end{array}
                                                     \right).\end{eqnarray*}
By direct computation we have
\begin{eqnarray*}&&(E_{11})^\dag_{M_{11},M_{11}}=\left(\begin{array}{cc}
             \frac{1}{4} & 0 \\
             0 & \frac{1}{12} \\
           \end{array}
         \right),
F_1(E)=\binom{-(E_{11})^\dag_{M_{11},M_{11}} E_{12}}{I_{K_2}}=\left(
                                                                \begin{array}{cc}
                                                                  -\frac{1}{4} & 0 \\
                                                                  -\frac{1}{6} & 0 \\
                                                                  1 & 0 \\
                                                                  0 & 1 \\
                                                                \end{array}
                                                              \right),
\\
&&F_1(E)^\#=S(M)^{-1}\cdot F_1(E)^*\cdot {\rm
diag}\left(M_{11},S(M)\right)=\left(
                                \begin{array}{cccc}
                                  -1 & -\frac{1}{3} & 1 & 0 \\
                                  0 & 0 & 0 & 1 \\
                                \end{array}
                              \right),
\\
&&F_1(E)^\# \cdot F_1(E)=\left(
                           \begin{array}{cc}
                             \frac{47}{36} & 0 \\
                             0 & 1 \\
                           \end{array}
                         \right), S(E)=E_{22}-E_{21}\cdot
                         (E_{11})^\dag_{M_{11},M_{11}}\cdot E_{12}=\left(
                                                                     \begin{array}{cc}
                                                                       \frac{7}{3} & 0 \\
                                                                       0 & 0 \\
                                                                     \end{array}
                                                                   \right),\\
&&Z_1\stackrel{def}{=}S(M)F_1(E)^\#\,F_1(E)=\left(
                                           \begin{array}{cc}
                                             \frac{47}{72} & 0 \\
                                             0 & 1 \\
                                           \end{array}
                                         \right),
Z_2\stackrel{def}{=}S(M)\big(F_1(E)^\# F_1(E)\big)^{-1}=\left(
                                           \begin{array}{cc}
                                             \frac{18}{47} & 0 \\
                                             0 & 1 \\
                                           \end{array}
                                         \right).
\end{eqnarray*}
Note that $\left((K_2)_{S(M)}\right)_{F_1(E)^\# F_1(E)}=(K_2)_{Z_1}$
and $\left((K_2)_{S(M)}\right)_{\big(F_1(E)^\#
F_1(E)\big)^{-1}}=(K_2)_{Z_2}$, so by (\ref{eqn:S-E-g}) we have
$$S(E)^g=S(E)^\dag_{Z_2,Z_1}=\left(
                               \begin{array}{cc}
                                 \frac{3}{7} & 0 \\
                                 0 & 0 \\
                               \end{array}
                             \right).
$$
Let $T=\left(
         \begin{array}{cc}
           -1 & -\frac{1}{3} \\
           0 & 0 \\
         \end{array}
       \right)\in {\cal L}\left((K_1)_{M_{11}},
(K_2)_{S(M)}\right)$ and $Z_3={\rm diag}\big(M_{11}, S(M)\big)$. As
$$F_1(E)^\#=(T,I_{K_2})
\in {\cal L}\left((K_1)_{M_{11}}\oplus (K_2)_{S(M)},
(K_2)_{S(M)}\right)={\cal L}\left(\left(K_1\oplus K_2\right)_{Z_3},
(K_2)_{S(M)}\right),$$ if we replace $H_1,H_2,H_3,A,B,M,N_1,L$ and
$N_2$ with $K_1,K_2,K_2, T, I_{K_2}, S(M), M_{11},0$ and $S(M)$
respectively, then we may apply Theorem~3.3 to get
$$\big(F_1(E)^\#\big)^\dag_{S(M),Z_3}=\left(
    \begin{array}{cc}
      T^\dag_{S(M),M_{11}}-D\widetilde{U} \\
      \widetilde{U} \\
    \end{array}
  \right),$$
where
\begin{eqnarray*}&&D=T^\dag_{S(M),M_{11}}=\left(
                                                                        \begin{array}{cc}
                                                                          -\frac{9}{11} & 0 \\
                                                                          -\frac{6}{11} & 0 \\
                                                                        \end{array}
                                                                      \right),\
\widetilde{S}=S(M)+D^* M_{11} D=\left(
                           \begin{array}{cc}
                             \frac{47}{22} & 0 \\
                             0 & 1 \\
                           \end{array}
                         \right),\\
                         &&
                         C=I_{K_2}-TT^\dag_{S(M),M_{11}}=\left(
                                \begin{array}{cc}
                                  0 & 0 \\
                                  0 & 1 \\
                                \end{array}
                              \right),\
                              C^\dag_{S(M),\widetilde{S}}=C,\\
                              \\
&&\widetilde{U}=C^\dag_{S(M),\widetilde{S}}+\left(I_{K_2}-C^\dag_{S(M),\widetilde{S}}\,
C\right)(\widetilde{S})^{-1}D^* M_{11} T^\dag_{S(M),M_{11}}=\left(
                                                 \begin{array}{cc}
                                                   \frac{36}{47} & 0 \\
                                                   0 & 1 \\
                                                 \end{array}
                                               \right).
\end{eqnarray*}
Therefore,
$$\big(F_1(E)^\#\big)^\dag_{S(M),Z_3}=\left(
                    \begin{array}{cc}
                      -\frac{9}{47} & 0 \\
                      -\frac{6}{47}& 0 \\
                      \frac{36}{47} & 0 \\
                      0 & 1 \\
                    \end{array}
                  \right).
$$
It follows from (\ref{eqn:X-R_E}) that
\begin{eqnarray*}&&
X_R(E)=I_{K_1\oplus K_2}-\big(F_1(E)^\#\big)^\dag_{S(M),Z_3}
\big(I_{K_2}-S(E)S(E)^g\big)F_1(E)^\#={\rm
diag}(1,1,1,0),\end{eqnarray*} hence by (\ref{equ:formula for A
dag-general}) we get
\begin{eqnarray*}&& B^\dag_{{\rm diag}\big(M_{11},S(M)\big),{\rm
diag}\big(N_{11},S(N)\big)}\\
&&=B^\#\cdot \left[{\rm
diag}\left((E_{11})^\dag_{M_{11},M_{11}},0\right)\,X_R(E)+F_1(E)\,S(E)^g\,F_1(E)^\#\right]=\left(
                                          \begin{array}{cc}
                                            (B^\dag)_{11} & (B^\dag)_{12} \\
                                            (B^\dag)_{21} & (B^\dag)_{22} \\
                                          \end{array}
                                        \right),
\end{eqnarray*}
where
\begin{eqnarray*}&& (B^\dag)_{11}=\left(
                                       \begin{array}{cc}
                                         \frac{4}{7} & \frac{1}{42} \\
                                         -\frac37 & \frac{1}{42} \\
                                       \end{array}
                                     \right), \quad (B^\dag)_{12}=\left(
                                                              \begin{array}{cc}
                                                                -\frac{1}{14} & 0 \\
                                                                -\frac{1}{14} & 0 \\
                                                              \end{array}
                                                            \right),\\&&(B^\dag)_{21}=\left(
                                                                                    \begin{array}{cc}
                                                                                      \frac{9}{14} & \frac{13}{28} \\
                                                                                      -\frac{3}{14} & \frac{5}{28} \\
                                                                                    \end{array}
                                                                                  \right),
\quad (B^\dag)_{22}=\left(
                \begin{array}{cc}
                  -\frac{9}{14} & 0 \\
                  \frac{3}{14} & 0 \\
                \end{array}
              \right).
\end{eqnarray*}
It follows from (\ref{eqn:A11})--(\ref{eqn:A22}) that
\begin{eqnarray*}&&(A^\dag_{MN})_{11}=\left(
                            \begin{array}{cc}
                              \frac17 & -\frac27 \\
                              -\frac{3}{14} & \frac{5}{28} \\
                            \end{array}
                          \right), \quad (A^\dag_{MN})_{12}=\left(
                                            \begin{array}{cc}
                                              \frac{3}{7} & 0 \\
                                              -\frac{11}{28} & 0 \\
                                            \end{array}
                                          \right),\\
&&(A^\dag_{MN})_{21}=\left(
           \begin{array}{cc}
             \frac{9}{14} & \frac{13}{28} \\
             -\frac{3}{14} & \frac{5}{28} \\
           \end{array}
         \right),\quad (A^\dag_{MN})_{22}=\left(
                                \begin{array}{cc}
                                  -\frac{9}{28} & 0 \\
                                  \frac{3}{28} & 0 \\
                                \end{array}
                              \right),\end{eqnarray*}
therefore,
$$A^\dag_{MN}=\left(
                \begin{array}{cccc}
                  \frac17& -\frac27 & \frac37 & 0 \\
                  -\frac{3}{14} & \frac{5}{28} & -\frac{11}{28} & 0 \\
                  \frac{9}{14} & \frac{13}{28} & -\frac{9}{28} & 0 \\
                  -\frac{3}{14} & \frac{5}{28} & \frac{3}{28} & 0 \\
                \end{array}
              \right).$$
\end{ex}

\vspace{2ex}

\vspace{8ex}

{\textbf Acknowledgements}

\vspace{2ex}

We sincerely thank Professor Richard A. Brualdi for his help, and the referee for his/her very useful comments and suggestions.

\vspace{2ex}
\bibliographystyle{plain}

\end{document}